\documentclass{article}

\usepackage{amsmath}
\usepackage{amssymb}
\usepackage{amsthm}
\usepackage{enumerate}
\usepackage{bbm}

\usepackage{multicol}
\usepackage{amsfonts}
\usepackage[colorlinks=true,citecolor=red,linkcolor=blue]{hyperref}

\newtheorem{theorem}{Theorem}[section]
\newtheorem{proposition}[theorem]{Proposition}
\newtheorem{corollary}[theorem]{Corollary}
\newtheorem{lemma}[theorem]{Lemma}

\newtheorem{rmk}[theorem]{Remark}

\newtheorem*{theorem*}{Theorem}
\newtheorem*{corollary*}{Corollary}

\newtheorem{definition}[theorem]{Definition}

\newcommand{\pp}{\mathfrak{p}}

\renewcommand{\O}{\mathcal {O}}
\newcommand{\F}{\mathbb{F}}

\newcommand{\ord}{\operatorname{ord}}

\newcommand{\ldv}{\mathcal{L}_{\rm div}}

\DeclareMathOperator{\ass}{as}
\DeclareMathOperator{\dib}{\operatorname{|}}
\DeclareMathOperator{\Div}{\operatorname{div}}

\newcommand{\prd}{\operatorname{Prod}}

\title{The Diophantine problem for addition and divisibility for rings of $S$-integers of quadratic imaginary extensions of $\mathbb{Q}$}
\author{Natalia Hormaz\'abal, Carlos  Mart\'inez-Ranero\footnote{Universidad de Concepci\'on, Chile (all two authors).}}

\begin{document}
\maketitle

\begin{abstract}
Let $K$ be a quadratic imaginary extension of $\mathbb{Q}$, let $S$ be a finite nonempty set of non archimedean places, and let $\O_{K,S}$ denote the ring of $S$-integers of $K$. We show that there is no algorithm  which solves the following problem. Given an arbitrary system of linear equations over the integers together with divisibility conditions on some of the variables, decide whether or not there exists a solution over  $\O_{K,S}$. This contrasts with  Lipshitz's result, which shows that such algorithm does exists for the ring of integers (i.e. $S=\emptyset).$ \footnote{ The  second named author was partially supported by  Proyecto VRID-Investigación  No. 220.015.024-INV. }
\end{abstract}


\section{Introduction and main results}
In this paper we  investigate the following problem. Given a commutative ring $R$.  Is there 
 an algorithm for the following decision problem:\\
 Given any sentence of the form 
\begin{equation} \label{eqn:intro}
\exists x_1,\dots, \exists x_n
\bigwedge_{i=1}^k f_i(x_1,\dots,x_n)=0\wedge
\bigwedge_{i=1}^m g_i(x_1,\dots,x_n) \ |\  h_i(x_1,\dots,x_n),
\end{equation}
\noindent where the $f_i$, $g_i$ and $h_i$ are linear polynomials with integer coefficients, decide whether or not  it has a solution over $R$.  

This line of research has a long history, primarily in connection with Hilbert's tenth problem and it is particularly interesting for rings with meaningful arithmetic   such as rings of integers, or more generally, rings of $S$-integers over global fields. We now provide a brief historical overview of the  most relevant results  to our investigation.  Lipshitz \cite{L77}, and in parallel Bel’tyukov \cite{B76}, showed that the above problem is decidable for $\mathbb{Z}$ and  for the ring of integers of a quadratic imaginary extension of $\mathbb{Q}.$ In contrast, Lipshitz \cite{L78b} showed that the  problem for the ring of integers $\mathcal{O}_K$ of any number field other than quadratic imaginary, is as complicated as Hilbert Tenth Problem for $\mathcal{O}_K$ (i.e. considering polynomials of arbitrary degree in the input). The latter is widely believed to be unsolvable. 

Complementing the above results, it was recently shown by Cerda-Romero and Mart\'inez-Ranero \cite{CM17} that if $S$ is a finite and non-empty set of primes, then the problem is undecidable for $\mathbb Z[S^{-1}]$. It also worth mentioning that the question for the ring of $S$-integers of a global function fields was completely settled by Martínez-Ranero, Utreras and Vidaux \cite{muv}.  

In this paper, we will extend the above results to the ring of $S$-integers of  a quadratic imaginary extension of $\mathbb{Q}$. In order to make this precise we need to introduce a few definitions. 

Let $K$ be a quadratic imaginary extension of $\mathbb{Q}$. For any finite nonempty set $S$ of non archimedean places of $K$  we define $\O_{K,S}$ to be the ring of $S$-integers $$\O_{K,S}:=\{\, x\in K\colon \ord_\pp(x)\geq 0, \rm{for\ all}\ \pp\notin S\,\}.$$
Let us consider the language $\mathcal{L}_{\Div}=\{=,0,1,+,\mid\}$ --- here $x\mid y$ will be interpreted as ``$x$ divides $y$''.

In terms of mathematical logic our main result can be stated as follows.

\begin{theorem}{\rm ( Theorem \ref{mainthm})}\label{mainund}
Multiplication is positive-existentially definable over the  $\ldv$-structure $\O_{K,S}$. In particular, the positive-existential theory of this structure is undecidable.
\end{theorem}

The  proof follows the usual pattern, but is technically much more involved.  First we show that being different from zero is positive-existentially
definable. This step turns out to be quite technical and requires a deep Theorem of
Lenstra \cite{Len77} from analytic number theory. Next we show that there is an abundance of units of arbitrary large "height". Finally, we provide a positive-existentially definition of the square
function on $\O_{K,S}$. This is the hardest part. The formulas we use are quite similar to
the ones used in \cite{CM17} and \cite{muv}. However, the proofs are more technically demanding as our rings are (usually) not PID and we also have to deal with archimedean places.

The structure of the rest of paper is as follows. In Section 2  we prove that being different from zero is positive-existentially definable. In Sections 3 we show that multiplication for units is definable. Section 4 is the technical hearth of the paper and we show the existence  of arbitrarily large units. Finally, in Section 5 we prove Theorem \ref{mainthm}.

\section{Definability of ``to be distinct''}
Let $K$ be a quadratic imaginary extension of $\mathbb{Q}$ and let $S$ be a finite nonempty set of non archimedean places of $K$, which we consider fixed for the rest of the paper.
The main goal of this section is to prove that the relation ``different from $0$'' is positive-existentially definable in the $\ldv$-structure $\O_{K,S}$.  In order to do this, we need to recall some results and a few definitions.

\begin{definition}
Let $M_K^*$ denote the set of all non-archimedean places of $K$, and let  $M\subseteq M_K^\star$. The upper and lower \emph{Dirichlet density} of $M$ are defined by 
$$
\delta_{+}(M)=\limsup\limits_{s\to 1^+}\frac{\sum_{\pp\in M} q_\pp^{-s}}{ \sum_{\pp\in M_K^\star} q_\pp^{-s}},
$$

$$
\delta_{-}(M)=\liminf\limits_{s\to 1^+}\frac{\sum_{\pp\in M} q_\pp^{-s}}{ \sum_{\pp\in M_K^\star} q_\pp^{-s}},
$$
 where $q_\pp$ is the size of the residue field at $\pp$. If $\delta_{+}(M)=\delta_{-}(M)$, then we denote by  $\delta(M)$  the common value and say that it has Dirichlet density.
\end{definition}

 We could not find a proper reference for the following easy lemma (see for instance \cite[Ex. 2.3]{Con04}). 

\begin{lemma}\label{lemmaconrad}
 The set $M^{\rm un}$ of non-archimedean places $\pp$ on $K$ unramified over $\mathbb{Q}$, and such that the residue field degree $f(\pp)=[k(\pp)\colon k(P)]$ is $1$ (where $P$ is the prime below $\pp$), has Dirichlet density $1$.
\end{lemma}

If $W$ is a finitely generated subgroup of $K^\times$, of (torsion-free) rank $r\geq 1$, and $\pp$ is a non-archimedean prime, we let $\psi_\pp\colon W\rightarrow k(\pp)^\times$ denote the quotient map to the unit subgroup of the residue field, and let $M(K,W)$ be the set of primes $\mathfrak p$ such that:
 \begin{enumerate}
 \item $v_{\mathfrak p}(w)=0$ for all $w\in W$, and
 \item the index of $\psi_\pp(W)$ in  $k(\pp)^\times$ divides $1$.
 \end{enumerate}     

In \cite[Statement 2.4]{Len77}, Lenstra proves in particular that: 

\begin{theorem}\label{thmlenstra}
The set $M(K,W)$ has upper Dirichlet density strictly less than $1$. 
\end{theorem}
In Lenstra's notation our $M(K,W)$ corresponds to $M(K,K,\{Id\},W,1)$. 
\begin{lemma}\label{lemlenstra} There exist a rational prime  $p\in \mathbb{N}$  and a non-zero  integer $0<b<p$, such that $px+b$  is never a unit in $ \O_{K,S}$ as $x$ varies over $ \O_{K,S}$.
\end{lemma}
\proof Let $M:=\{\pp\in M_K^*\setminus S: [k(\pp)^\times: \psi_\pp(\O_{K,S}^\times)]>1\}.$ It follows from Theorem \ref{thmlenstra} that $\delta_-(M)>0$. We claim that $M\cap M^{\rm un}\ne \emptyset$. Indeed, if $M\cap M^{\rm un}= \emptyset,$ then $\delta_-(M\cap M^{\rm un})\geq \delta_-(M)+\delta_-(M^{\rm un})>1$ which is impossible. 
  
Pick $\pp \in M\cap M^{\rm un}$ and let $p\in \mathbb{N}$ be the rational prime generating the ideal $\pp\cap \mathbb{Z}.$  Now choose an integer $0<b<p$ such that ${\rm res}_\pp(b)\in k(\pp)^\times\setminus \psi_\pp(\O_{K,S}^\times)$. This is possible since $k(\pp)^\times=\F_p^\times$ and $[k(\pp)^\times: \psi_\pp(\O_{K,S}^\times)]>1$.  It follows from the construction  that  $p$ and $b$ are as required. 
\endproof
We are now ready to prove the main result of the section.
\begin{lemma}\label{dif}
The relation $\neq$ is positive-existentially definable in the $\ldv$-structure $\O_{K,S}$.
\end{lemma}
\proof
Let $p$ and $b$ be given as in  Lemma \ref{lemlenstra}. The formula
$$
\psi_{\neq}(y):\exists A,B,x(y\mid A\wedge px+b\mid  B\wedge A+B=1)
$$
positive-existentially defines the relation ``$y\neq 0$'' in $\O_{K,S}$.

First note that the formula $\psi_{\neq}(y)$ translates to  ``There exist $r,s,x\in \O_{K,S}$ such that $ry+s(px+b)=1$'' in $ \O_{K,S}$.   

If  $y=0$, then the formula is false, since by Lemma \ref{lemlenstra}, $px+b$ is never a unit in $\O_{K,S}$. 

For the other implication suppose $y\neq 0$. We may assume that $y\in \O_K$ (otherwise, multiply the $r$ that we will find by a denominator of $y$).  Let $\pp_1^{k_1}\cdots \pp_m^{k_m}$ be the prime factorization  of the ideal $y\O_{K}$, and let $P_1,\dots, P_m$ be the primes below $\pp_1,\dots, \pp_m,$ respectively.   Since $p$ and $b$ are relatively prime, by Dirichlet's  theorem there exists $x$ such that $qx+b$ is prime, and coprime with  $P_1,\dots, P_m$. It follows that the ideals $y\O_K$ and $(qx+b)\O_K$ are comaximal. Hence, there are $r,s \in \O_K$ so that $ry+s(qx+b)=1$ as required.     \endproof

\section{Defining the multiplication function restricted to units}
In this section we show that the multiplication function restricted to units is positive-existentially definable in $\O_{K,S}$. Before proceeding any further, we need to introduce some notation. 
\begin{itemize}
\item $p_1,\dots, p_k$ are the prime elements in $\mathbb{Z}$ below the primes of $S$.
\item $s=\# S$. 
\item $u_K =\#\O_{K,S}^\times$.
\item for $x,y\in \O_{K,S},$ we abbreviate the formula $x\dib y \wedge y\dib x$ by $as(x,y).$
\item $I=\{0,\dots, 4\}^k$.
\item If $\alpha=(\alpha_1,\dots,\alpha_k)\in \mathbb N^k$ is a multi-index, we write $p^\alpha=\prod_{i=1}^kp_i^{\alpha_i}$. 
\item For $i=1,\dots, u_K+1$, let  $q_{i}\in\mathbb{Z}$ be pairwise different rational primes not below any prime of $S$.
\end{itemize}

\begin{lemma}\label{fakepigeonhole}
For every $x,y \in \O^\times_{K,S}$, there exist a multi-index $\beta\in I$  such that $v_\pp(p^\beta x)\ne 0$ and $v_\pp(p^\beta z)\ne v_\pp(y)$ for every prime $\pp$ in $S$.
\end{lemma}
\proof
Given $x,y \in \O^\times_{K,S}$, for each  $1\leq i \leq k$, choose 
$$
\beta_i\in\{0,1,\dots,4\}\setminus \left(\left\{-\frac{v_\pp(x)}{e(\pp:p_i)}\colon \pp \dib p_i\right\}\cup\left\{\frac{v_\pp(z)-v_\pp(y)}{e(\pp:p_i)}\colon \pp  \dib   p_i\right\}\right),
$$
arbitrarily. Observe that this is possible since there are at most two primes above any of the $p_i$'s. It is straightforward to verify that $\beta:=(\beta_1,\dots,\beta_k)$ is as required.
\endproof

\begin{proposition}\label{produnits}
The set $\prd_u=\{(x,y,z)\colon x,y,z \in \O^\times_{K,S},\ z=xy\}$ is positive-existentially definable in $\O_{K,S}$.
\end{proposition}
\proof
We shall prove that the following  formula works.
$$
\prd_u(x,y,z)\colon x\mid 1\wedge y\mid 1\wedge z\mid 1 \wedge  \bigwedge_{\substack{\alpha\in I\\i=1,\dots,u_K+1 }} \ass(p^{\alpha} x+q_i,p^{\alpha}z+q_iy)
$$
On one hand, if $z=xy$ for some $x,y,z\in\O^\times_{K,S}$, then $\prd_u(x,y,z)$ holds trivially in $\O_{K,S}$. 

On the other hand, suppose  that $\prd_u(x,y,z)$ holds in $\O_{K,S}$. Choose $\beta\in I$ a multi-index satisfying the conclusions Lemma \ref{fakepigeonhole}. Let $u_{i}$ be units such that
\begin{equation}\label{equprin3-1}
p^{\beta}x+q_i=u_{i}(p^{\beta}z+q_iy).
\end{equation}
Notice that by our choice of $q_i$'s and $\beta$, we have that $v_\pp(p^{\beta}x)\ne v_\pp(q_i)$ and $v_\pp(p^{\beta}z)\ne v_\pp(q_iy)$ for all $\pp\in S$ and all $i\in \{1,\dots,u_K+1\}$. Thus, 
$$
\begin{aligned}
v_\pp(u_{i})&=v_\pp(u_{i}(p^{\beta}x+q_i))-v_\pp(p^{\beta}z+q_iy)\\
&=\min \{v_\pp(p^{\beta}x),v_\pp(q_i)\}-\min \{v_\pp(p^{\beta}z),v_\pp(q_iy)\},
\end{aligned}
$$
and the latter expression does not depend on $i$, because by definition of $q_i$ we have $v_\pp(q_i)=0$ for all $\pp\in S$ and all $i\in \{1,\dots,u_K+1\}$. In particular, we have $v_\pp (u_{i})=v_\pp (u_{j})$ for every prime $\pp$ in $S$ and every $i$ and $j$. This implies that $u_i$ and $u_j$ have the same zeros and poles, hence they differ by a unit in $\O_K$. Since there are $u_K$  units in $\O_K$, using the pigeonhole principle we obtain that there exists $i\ne j$ so that $u_i=u_{j}=u$. 

From Equation \eqref{equprin3-1} for $i$ and $j$, after subtracting the sides of the equations, we obtain $1=uy$.  Substituting in the first equation we get $z=xy$. This concludes the proof of the Proposition. 
\endproof
\section{Abundance of large units }
In this section we introduce a notion of relative largeness between non-zero elements and units of $\O_{K,S}$. Roughly speaking, we want to express the fact that  for each non-zero element there is a unit of greater "height" using divisibility conditions.


\begin{lemma}\label{lemunitinforder}
For every $x\in O^*_{K,S}$ there exists a unit $\varepsilon$ of infinite order such that $x \dib \varepsilon-1$.  
\end{lemma}
\proof
Let $x\in O^*_{K,S}$ be given. Let $\pp_1^{\alpha_1}\cdots \pp_\ell^{\alpha_\ell}$ be the factorization of the fractional ideal $x^{h_K}\O_K$ where $h_K$ denotes the ideal class number. For each, $1\leq i\leq \ell$, choose $x_i$ such that $x_i\O_K=\pp_i^{|\alpha_i|}$, this is possible since $h_K| \alpha_i.$ Set $y=\prod_{\pp_i\notin S} x_i$. Choose $\pp\in S$ (this is where we need that  $S$ is nonempty) and let $\varepsilon_0$ be a generator of $\pp^{h_K}$. Since $\O_K/(y)$ is finite, there are $m>n>0$ such that $\varepsilon^m_0\equiv \varepsilon_0^n \mod (y)$. Using that $\varepsilon_0$ is not a zero-divisor in the quotient we infer that $y|\varepsilon_0^{m-n}$. It follows that $\varepsilon:=\epsilon_0^{m-n}$ is as required.  
\endproof

\begin{rmk}
	Note that since the torsion part of  $\O^\times_{K,S}$ is finite, there is a positive existential $\mathcal{L}_{\Div}$-formula $\varphi_\infty(x)$ so that $\varphi_\infty(u)$ holds in $\O_{K,S}$ if and only if $u$ is a unit of infinite order.
\end{rmk}

Recall that, for each nonzero prime $\pp$ of $\O_K$, $k(\pp):=\O_K/\pp$ denotes the residue field at $\pp$ and $q_\pp:=\#k(\pp)$. We also recall the product formula $|a|^2=\prod_{\pp\in M_K^*}q_\pp^{v_\pp(a)}$ for each $a\in K$.

\begin{lemma}\label{keylemma}
There is a constant $0<C<1 $ such that given  any $x\in \mathcal O_{K,S}^*$, there exists $a\in  O_{K,S}^*$  and $b\in \O^*_{K}$ such that $x=\frac{a}{b}$, $v_\pp(a)> -h_K$ for every $\pp\in S$, and $C<|a|$.  Moreover, for every $\pp\in S$, if $v_\pp(a)\geq 0,$ then $v_\pp(b)=0$.
\end{lemma}
\begin{proof}
   Let  $C^2=\prod\limits_{\pp\in S} q_\pp^{-h_K}$ where $h_K$ denotes the ideal class group. Let $D(x):=\{\pp\in S: v_\pp(x)<0\} $. If $D(x)=\emptyset$, then set $b=1$ and $a=x$. Next, suppose $D(x)\ne \emptyset$. For each $\pp\in D(x)$, let $\beta_\pp$ be the maximal integer $\leq -v_\pp(x)$ such that $\pp^{\beta_\pp}$ is a principal ideal. Choose $b_\pp$ such that $b_\pp \O_K=\pp^{\beta_\pp}$. Finally, let $b=\prod\limits_{\pp \in D(x)} b_\pp $ and let $a:=bx$. Notice that by construction $v_\pp(a)>-h_k$ and   for every $\pp\in S$, if $v_\pp(a)\geq 0,$ then $v_\pp(b)=0$. We claim that $|a|> C.$
   Indeed, by the product formula

   $$
   |a|^2=\prod_{\pp\in M_K^*}q_\pp^{v_\pp(a)}=\prod_{\pp\in M_K^*\setminus S}q_\pp^{v_\pp(a)}\prod_{\pp\in S}q_\pp^{v_\pp(a)}
   $$
   $$
   > \prod_{\pp\in S}q_\pp^{-h_K}=C^2.
   $$
   Thus, $|a|> C$ as required.
\end{proof}


The following Lemma makes precise the idea that the unit $\varepsilon$ in Lemma \ref{lemunitinforder} is much larger than $x$. Let $q$ be a rational prime  bigger than $4C^{-2}$ and not below any of the primes in $S$.

\begin{lemma}\label{lemC2}
Let $x\in \O^*_{K,S} $ be such that $v_\pp(x)\ne 0$ for any $\pp\in S$, and let $\varepsilon\in \O_{K,S}^\times$ be a unit of infinite order such that $x\pm 1 | \varepsilon-1$, and $q\dib \varepsilon-1$. If we write $x=\frac{a}{b}$ and $\varepsilon=\frac{u}{v}$, where $a,b$ and $u,v$ are as in Lemma \ref{keylemma}, then  
$$
\max\{\frac{2}{C},|a|,|b|\}\leq \max\{|u|^2,|v|^2\}
$$
\end{lemma}
\proof Write $x=\frac{a}{b}$, and $\varepsilon=\frac{u}{v}$, where $a,b$ and $u,v$ are as in the  Lemma \ref{keylemma}.\\
 First of all, we shall show that $q\dib \varepsilon-1$, implies that $\frac{1}{C}\leq \max\{|u|^2,|v|^2\}$. Notice that, since $v$ is a unit 
$q|\varepsilon-1$ if and only if $q|u-v$. Let $r:=\frac{u-v}{q}$. For each $\pp\in S$, $v_\pp(r)=v_\pp(u-v)-v_\pp(q)=v_\pp(u-v)> -h_K.$ Thus, by the product formula, $|r|> C.$ It follows, from our assumptions on $q$, and the triangle inequality that $$2\max\{|u|,|v|\}\geq |u-v|=|q||r|> qC>\frac{4}{C}.$$

Next, we show that $x\pm 1 | \varepsilon-1$ implies $C|a\pm b|\leq |u-v|$.
We will focus on the $+$ case as the other one is analogous. Proceeding as before, we have that $a+b|u-v$. Let $c:=\frac{u-v}{a+b}$. We claim that $v_\pp(c)>-h_K$ for $\pp\in S$. Indeed, first observe that  $v_\pp(u-v)\geq \min\{v_\pp(u),v_\pp(v)\}>-h_K.$ Next, we need an upper bound for $v_\pp(a+b)$ for $\pp\in S$. In order to do this, we proceed by cases. On one hand, if $v_\pp(x)>0,$ then $v_\pp(b)=0$ and thus $v_\pp(a+b)=0$. On the other hand, if $v_\pp(x)<0$, then either $v_\pp(a)=0$ and $v_\pp(b)>0$ or $v_\pp(a)<0 $ and $v_\pp(b)\geq 0$. Hence, in either case $v_\pp(a+b)\leq 0$. It follows that $v_\pp(c)=v_\pp(u-v)-v_\pp(a+b)>-h_K$. Using the product formula again we have that 

$$|u-v|^2= \prod_{\pp\in M_K^*}q_\pp^{v_\pp(u-v)}=\prod_{\pp\in M_K^*}q_\pp^{v_\pp(c(a+b))}=|a+b|^2\prod_{\pp\in M_K^*\setminus S}q_\pp^{v_\pp(c)}\prod_{\pp\in S}q_\pp^{v_\pp(c)}$$ 
$$
   > |a+b|^2\prod_{\pp\in S}q_\pp^{-h_K}=|a+b|^2C^2.
   $$
Thus, $|u-v|>C|a+ b|$. Analogously, we obtain $|u-v|>C|a-b|$.
Finally, using the triangle inequality we obtain that
$$
2|a|\leq |a+b|+|a-b|< \frac{2}{C}|u-v|\leq \frac{4}{C}\max\{|u|,|v|\}
$$
Hence, $|a|< \frac{2}{C}\max\{|u|,|v|\}$. Proceeding in an analogous way we also have that $|b|<\frac{2}{C}\max\{|u|,|v|\}$. Combining these with the previous inequality $\max\{|u|,|v|\}>\frac{2}{C}$. We infer that 
$\max\{\frac{2}{C},|a|,|b|\}\leq \max\{|u|^2,|v|^2\}$.  \endproof
\section{Definability of the square function}
We are ready to prove the main theorem of the paper. The proof follows the same pattern as before, though much harder. Recall that 
\begin{itemize}
\item $p_1,\dots, p_k$ are the prime elements in $\mathbb{Z}$ below the primes of $S$.
\item and let $J=\{0,\dots, 6\}^k$.
\end{itemize}

\begin{theorem}\label{mainthm}
	The set 
	$$
	SQ=\{(x,y)\in \O_{K,S}\colon y=x^2\}
	$$
	is positive-existentially definable in the   $\ldv$-structure $\O_{K,S}.$
\end{theorem}
\begin{proof} We claim that the following formula $\varphi_{\rm sq}(x,y)$ defines the set $SQ$.
\begin{equation}
(x=0\wedge y=0)\vee \bigvee_{\substack{\alpha\in J}}(x=\pm p^{-\alpha}\wedge y=p^{-2\alpha})
\end{equation}

\begin{equation}
\vee  \exists \varepsilon(\varepsilon \dib 1 \wedge \varphi_\infty(\varepsilon)) \wedge \psi(x,y)
\end{equation}

where $\psi(x,y)$ is the conjunction of the following formulas:
$$
\psi_1(x,\varepsilon): \bigwedge_{\substack{\alpha\in J}} p^{\alpha}x\pm 1\dib \varepsilon-1
$$

$$
\psi_2(y,\varepsilon): \bigwedge_{\substack{\alpha\in J}} p^{2\alpha}y\pm 1\dib \varepsilon-1
$$

$$
\psi_3(\varepsilon): q \dib \varepsilon-1
$$

$$
\psi_4(\varepsilon_2): \varepsilon_1=\varepsilon^{17}
$$

$$
\psi_5(x,y,\varepsilon_1): \bigwedge_{\substack{\alpha\in J}} p^{\alpha}x\pm \varepsilon_{1} \dib p^{2\alpha}y-\varepsilon_1^2.
$$
First we show that if $y=x^2$, then $\varphi_{\rm sq}(x,y)$ holds in $\O_{K,S}$. If $x=0$ or $x=\pm p^{-\alpha}$ for $\alpha\in I$, then clearly $\varphi_{\rm sq}(x,y)$ holds. So we may assume $x\ne 0$ and  $x\ne \pm p^{-\alpha}$ for any $\alpha\in I$, but in this case the result follows immediately from Lemma \ref{lemunitinforder}.
\\

Let us now assume that $\varphi_{\rm sq}(x,y)$ holds. Without loss of generality we may also assume that $x\ne0,$ and  $x\ne \pm p^{-\alpha}$ for any $\alpha\in I$. Fix $\varepsilon$  so that $\psi(x,y)$ holds. 

We will break the rest of the proof into several Claims.\\

\noindent\textbf{Claim A:} There exists a multi index $\alpha\in J$ such that $v_\pp(p^\alpha x)\ne 0,$ $v_\pp(p^{2\alpha} y)\ne 0$ and  $v_\pp(p^\alpha x)\ne v_\pp (\varepsilon_1)$ for all $\pp\in S$. \\
\textit{Proof of Claim A.}
Given $x,y \in \O^*_{K,S}$, for each  $1\leq i \leq k$,  choose 
$$
\alpha_i\in \{0,\dots,6\}\setminus \left(\left\{\frac{-v_\pp(x)}{e(\pp:p_i)}\colon \pp \dib p_i\right\}\cup\left\{\frac{-v_\pp(y)}{2e(\pp:p_i)}\colon \pp  \dib   p_i\right\}\cup \left\{\frac{-v_\pp(\varepsilon_1)}{e(\pp:p_i)}\colon \pp  \dib   p_i\right\} \right),
$$
arbitrarily. Observe that this is possible since there are at most two primes
above any of the $p_i$’s. It follows that $\alpha:=(\alpha_1,\dots,\alpha_k)$ is as required. $\hfill\square$

From now on, in order to simplify notation, we will write $X$ and $Y$ instead of $p^\alpha x$ and $p^{2\alpha}y$. We also write $X=\frac{a}{b}, Y=\frac{c}{d}$ and $\varepsilon=\frac{u}{v}$ where $a,b, c,d, $ and $u,v$ are chosen as in Lemma \ref{keylemma}.

Since $\psi_1(x,\varepsilon)$, $\psi_2(y,\varepsilon)$ and $\psi_3(\varepsilon)$ hold, it follows  from Lemma \ref{lemC2} that
\begin{equation}\label{lasteq}
\begin{split}
\max\{\frac{2}{C},|a|,|b|,|c|,|d|\}\leq \max\{|u|^2,|v|^2\}
\end{split}
\end{equation}


From $\psi_5$,  we have 
$$
X\pm  \varepsilon^{17}\ |\ Y-\varepsilon^{34},
$$ 
and  trivially
$$
X\pm  \varepsilon^{17}\ |\  X^2-\varepsilon^{34},
$$ 
hence $X\pm \varepsilon^{17} \  | \ Y-X^2$. 
Aiming towards a contradiction, suppose that $Y\ne X^2$.  It follows, from the previous equation, that 
\begin{equation}\label{eqtocontradict}
av^{17}\pm u^{17}b\ \dib \ b^2c-a^2d,
\end{equation}

We will obtain our desired contradiction by showing that Equation \eqref{eqtocontradict} does not hold. Intuitively, the left hand size is much bigger than the right hand side.

Let $e^{\pm}:=\frac{b^2c-a^2d}{av^{17}\pm u^{17}b}$. First we shall obtain an upper bound for $|b^2c-a^2d|$ and a lower bound for its poles. For this, notice that $|b^2c-a^2d|\leq 2\max\{|a|^3,|b|^3,|c|^3,|d|^3\}$ and also that 
\begin{equation}\label{eqpoles}
    v_\pp(b^2c-a^2d)\geq \min\{v_\pp(b^2c), v_\pp(a^2d)\}>-2h_K \ \ \ {\rm for}\ \pp\in S.
\end{equation}


Next we shall obtain  lower bounds for  $|e^{\pm}|$.   In order to do this, we need to find  lower bounds for the poles of $e^{\pm}$. In turn, this requires lower bounds for the poles of $b^2c-a^2d$ and upper bounds for the zeros of $av^{17}\pm u^{17}b$ in $S$. 

\bigskip
\noindent\textbf{Claim B:} For each $\pp\in S$, $v_\pp(av^{17}\pm u^{17}b)\leq h_K+v_\pp(a)+v_\pp(b)$.  

\noindent\textit{Proof of Claim B.}  Fix $\pp\in S$. Notice that, by our assumptions, $v_\pp(av^{17}\pm u^{17}b)=\min\{v_\pp(av^{17}), v_\pp(bu^{17})\}$.
Now we proceed by cases:

\begin{itemize}
   
\item If $v_\pp(a)>0$ and $v_\pp(v)>0$, then $v_\pp(b)=0$ and $v_\pp(u)\leq 0$. Hence, $\min\{v_\pp(av^{17}), v_\pp(bu^{17})\}\leq v_\pp(bu^{17})\leq 0$.

\item If $v_\pp(a)>0$ and $v_\pp(v)= 0$, then $v_\pp(b)=0$. Hence, $\min\{v_\pp(av^{17}), v_\pp(bu^{17})\}\leq v_\pp(av^{17})= v_\pp(a).$

\item If $v_\pp(a)=0$ and $v_\pp(v)>0$, then $v_\pp(b)>0$ and $v_\pp(u)\leq 0$. Hence, $\min\{v_\pp(av^{17}), v_\pp(bu^{17})\}\leq v_\pp(bu^{17})\leq v_\pp(b).$

\item If $v_\pp(a)=0$ and $v_\pp(v)=0$, then 
$\min\{v_\pp(av^{17}), v_\pp(bu^{17})\}\leq v_\pp(av^{17})=0.$

\item If $v_\pp(a)<0 $ and $v_\pp(v)>0$, then $v_\pp(u)\leq 0$ and $v_\pp(b)\geq 0$. Hence,
 $\min\{v_\pp(av^{17}), v_\pp(bu^{17})\}\leq v_\pp(bu^{17})\leq v_\pp(b).$

\item If $v_\pp(a)<0$ and $v_\pp(v)=0$, then  $\min\{v_\pp(av^{17}), v_\pp(bu^{17})\}\leq v_\pp(av^{17})<0$.

\end{itemize}

Thus it follows that   $$v_\pp(av^{17}\pm u^{17}b)\leq \max\{0, v_\pp(a),v_\pp(b)\}\leq h_K+v_\pp(a)+v_\pp(b).$$
$\hfill\square$

Combining Equation \ref{eqpoles} and Claim B, we infer that 
$$v_\pp(e^{\pm})=v_\pp(b^2c-a^2d)-v_\pp(av^{17}\pm bu^{17})\geq -2h_K-h_K-v_\pp(a)-v_\pp(b).$$
 
Using the product formula once more we have
$$
|e^{\pm}|^2=\prod\limits_{\pp\in M_K^*} q_\pp^{v_\pp(e^{\pm})}\geq \prod\limits_{\pp\in S} q_\pp^{-3h_K-v_\pp(a)-v_\pp(b)}\geq $$  $$ \left (\prod\limits_{\pp\in M_K^*\setminus S} q_\pp^{-v_\pp(a)-v_\pp(b)}\right ) \prod\limits_{\pp\in S} q_\pp^{-3h_K-v_\pp(a)-v_\pp(b)}=\frac{C^6}{|ab|^2}. 
$$

Observe that, by the triangle inequality,  

 $$2|b^2c-a^2d|=|e^+||av^{17}+ bu^{17}|+|e^-||av^{17}- bu^{17}|\geq 2\frac{C^3|av^{17}|}{|ab|}$$

and 

 $$2|b^2c-a^2d|=|e^+||av^{17}+ bu^{17}|+|e^-||-av^{17}+ bu^{17}|\geq 2\frac{C^3|bu^{17}|}{|ab|}$$.

 Thus, we obtain that $|b^2c-a^2d|\geq \max\left\{\frac{C^3|u^{17}|}{|a|},\frac{C^3|v^{17}|}{|b|}\right \}.$

 Using Equation \ref{lasteq}, it follows that 

$$
\max\left\{\frac{C^3|u^{17}|}{|a|},\frac{C^3|v^{17}|}{|b|}\right \}\geq \max\{|u^{9}|,|v|^{9}\}
$$

On the other hand, $|b^2c-a^2d|\leq 2\max\{|a|^3,|b|^3,|c|^3,|d|^3\}\leq \max\{|u|^8,|v|^8\}.$ Therefore,

$$
\max\{|u|^8,|v|^8\}\geq \max\{|u^{9}|,|v|^{9}\}
$$
which is a contradiction. This concludes the proof of the Theorem. \end{proof}

As a corollary we obtain the undecidability of $\rm{Th}^{+\exists }(\O_{K,S})$.
\begin{corollary}
	The positive-existential theory of the $\ldv$-structure $\O_{K,S}$ is undecidable.
\end{corollary}
\proof This follows immediately from the undecidability of $\O_{K,S}$ in the language of rings (see \cite{S07}). \endproof

%
%

\noindent Natalia Hormaz\'abal\\
Email: nhormazabal2017@udec.cl\\

\noindent Carlos A. Mart\'inez-Ranero\\
Email: cmartinezr@udec.cl\\
Homepage: www2.udec.cl/~cmartinezr\\

\noindent Same address: \\
Universidad de Concepci\'on, Concepci\'on, Chile\\
Facultad de Ciencias F\'isicas y Matem\'aticas\\
Departamento de Matem\'atica\\

\end{document}